\documentclass[11pt]{article}
\usepackage[margin=1in]{geometry} 
\usepackage{amssymb,amsfonts,amsmath,bbm,mathrsfs,stmaryrd,mathtools}
\usepackage{xcolor}
\usepackage{url}

\usepackage{enumerate}

\usepackage{hyperref}
\hypersetup{colorlinks,
             linkcolor=black!75!red,
             citecolor=blue,
             pdftitle={},
             pdfproducer={pdfLaTeX},
             pdfpagemode=None,
             bookmarksopen=true
             bookmarksnumbered=true}

\usepackage{tikz}
\usetikzlibrary{arrows,calc,decorations.pathreplacing,decorations.markings,intersections,shapes.geometric,through,fit,shapes.symbols,positioning,decorations.pathmorphing}

\usepackage{braket}

\usepackage[amsmath,thmmarks,hyperref]{ntheorem}
\usepackage{cleveref}

\creflabelformat{enumi}{#2(#1)#3}

\crefname{section}{Section}{Sections}
\crefformat{section}{#2Section~#1#3} 
\Crefformat{section}{#2Section~#1#3} 

\crefname{subsection}{\S}{\S\S}
\AtBeginDocument{%
  \crefformat{subsection}{#2\S#1#3}%
  \Crefformat{subsection}{#2\S#1#3}%
}

\crefname{subsubsection}{\S}{\S\S}
\AtBeginDocument{%
  \crefformat{subsubsection}{#2\S#1#3}%
  \Crefformat{subsubsection}{#2\S#1#3}%
}

\theoremstyle{plain}

\newtheorem{lemma}{Lemma}[section]
\newtheorem{proposition}[lemma]{Proposition}

\newtheorem{theorem}[lemma]{Theorem}

\theoremstyle{nonumberplain}
\newtheorem{theoremN}{Theorem}

\theoremstyle{plain}
\theorembodyfont{\upshape}
\theoremsymbol{\ensuremath{\blacklozenge}}

\newtheorem{definition}[lemma]{Definition}

\newtheorem{remark}[lemma]{Remark}

\crefname{definition}{definition}{definitions}
\crefformat{definition}{#2definition~#1#3} 
\Crefformat{definition}{#2Definition~#1#3} 

\crefname{ex}{example}{examples}
\crefformat{example}{#2example~#1#3} 
\Crefformat{example}{#2Example~#1#3} 

\crefname{remark}{remark}{remarks}
\crefformat{remark}{#2remark~#1#3} 
\Crefformat{remark}{#2Remark~#1#3} 

\crefname{convention}{convention}{conventions}
\crefformat{convention}{#2convention~#1#3} 
\Crefformat{convention}{#2Convention~#1#3} 

\crefname{notation}{notation}{notations}
\crefformat{notation}{#2notation~#1#3} 
\Crefformat{notation}{#2Notation~#1#3} 

\crefname{table}{table}{tables}
\crefformat{table}{#2table~#1#3} 
\Crefformat{table}{#2Table~#1#3}

\crefname{lemma}{lemma}{lemmas}
\crefformat{lemma}{#2lemma~#1#3} 
\Crefformat{lemma}{#2Lemma~#1#3} 

\crefname{proposition}{proposition}{propositions}
\crefformat{proposition}{#2proposition~#1#3} 
\Crefformat{proposition}{#2Proposition~#1#3} 

\crefname{corollary}{corollary}{corollaries}
\crefformat{corollary}{#2corollary~#1#3} 
\Crefformat{corollary}{#2Corollary~#1#3} 

\crefname{theorem}{theorem}{theorems}
\crefformat{theorem}{#2theorem~#1#3} 
\Crefformat{theorem}{#2Theorem~#1#3} 

\crefname{enumi}{}{}
\crefformat{enumi}{(#2#1#3)}
\Crefformat{enumi}{(#2#1#3)}

\crefname{assumption}{assumption}{Assumptions}
\crefformat{assumption}{#2assumption~#1#3} 
\Crefformat{assumption}{#2Assumption~#1#3} 

\crefname{equation}{}{}
\crefformat{equation}{(#2#1#3)} 
\Crefformat{equation}{(#2#1#3)}

\numberwithin{equation}{section}

\theoremstyle{nonumberplain}
\theoremsymbol{\ensuremath{\blacksquare}}

\newtheorem{proof}{Proof}
\newcommand\pf[1]{\newtheorem{#1}{Proof of \Cref{#1}}}

\newcommand\bC{{\mathbb C}}

\newcommand\bR{{\mathbb R}}
\newcommand\bS{{\mathbb S}}

\newcommand\bZ{{\mathbb Z}}

\newcommand\cC{{\mathcal C}}

\DeclareMathOperator{\id}{id}
\DeclareMathOperator{\End}{\mathrm{End}}
\DeclareMathOperator{\Aut}{\mathrm{Aut}}

\newcommand{\cat}[1]{\textsc{#1}}

\newcommand{\qedhere}{\mbox{}\hfill\ensuremath{\blacksquare}}

\title{Naturality and innerness for morphisms of compact groups and (restricted) Lie algebras} \author{Alexandru Chirvasitu}

\begin{document}

\date{}

\newcommand{\Addresses}{{
  \bigskip
  \footnotesize

  \textsc{Department of Mathematics, University at Buffalo, Buffalo,
    NY 14260-2900, USA}\par\nopagebreak \textit{E-mail address}:
  \texttt{achirvas@buffalo.edu}

}}

\maketitle

\begin{abstract}
  An extended derivation (endomorphism) of a (restricted) Lie algebra $L$ is an assignment of a derivation (respectively) of $L'$ for any (restricted) Lie morphism $f:L\to L'$, functorial in $f$ in the obvious sense. We show that (a) the only extended endomorphisms of a restricted Lie algebra are the two obvious ones, assigning either the identity or the zero map of $L'$ to every $f$; and (b) if $L$ is a Lie algebra in characteristic zero or a restricted Lie algebra in positive characteristic, then $L$ is in canonical bijection with its space of extended derivations (so the latter are all, in a sense, inner). These results answer a number of questions of G. Bergman.

  In a similar vein, we show that the individual components of an extended endomorphism of a compact connected group are either all trivial or all inner automorphisms.
\end{abstract}


\vspace{.5cm}



\section*{Introduction}

This note was prompted by a number of questions posed in \cite{bg}. That paper revolves around the notion of {\it innerness} for automorphisms, endomorphisms or other classes of maps between algebraic structures.

One intriguing piece of insight is that innerness (whatever it means it any given context) is automatic for {\it extended} morphisms (auto, endo, etc.): per \cite[Definition 4]{bg}, an extended endomorphism of an object $c\in \cC$ of a category is an endomorphism of the forgetful functor
\begin{equation}\label{eq:fgt}
  c\downarrow \cC\ni (c\to d)\xmapsto{\quad U_{c,\cC}\quad}d\in \cC
\end{equation}
from the {\it comma category} $c\downarrow \cC$ consisting of morphisms in $\cC$ with domain $c$ (\cite[\S II.6]{mcl}, \cite[Exercise 3K]{ahs}); the same goes for {\it auto}morphisms. The paradigmatic results (\cite[Theorems 1 and 2, Corollary 3]{bg}) say that for the category $\cat{Gp}$ of groups and $G\in \cat{Gp}$
\begin{itemize}
\item The morphism attaching to $g\in G$ the natural automorphism of $U_{G,\cat{Gp}}$ (notation as in \Cref{eq:fgt}) operating as conjugation by $f(g)$ for any
  \begin{equation}\label{eq:samplef}
    \left(G\xrightarrow{\quad f\quad}H\right)\in G\downarrow \cat{Gp}
  \end{equation}
  is an isomorphism $G\cong \mathrm{Aut}(U_{G,\cat{Gp}})$.
\item The only other extended {\it endo}morphism of $U_{G,\cat{Gp}}$ is the one operating trivially on $H$ for every \Cref{eq:samplef}.   
\end{itemize}

The specific questions that motivated this work are concentrated in \cite[\S 8]{bg}, which retraces some of this in the context of (restricted \cite[\S V.7 Definition 4]{jc}) Lie algebras and their {\it derivations}. The category $\cC$ of \Cref{eq:fgt} is now that of Lie algebras over a field (perhaps restricted, when in positive characteristic), and along with endomorphisms of \Cref{eq:fgt} one similarly considers {\it extended derivations} \cite[Definition 10]{bg} of a Lie algebra $L$:
\begin{itemize}
\item a derivation $\partial_f$ of the Lie algebra $L'$ for every morphism $f:L\to L'$;
\item with the $\partial_f$ satisfying the obvious compatibility conditions, analogous to those required of extended morphisms.
\end{itemize}

With that in place,
\begin{enumerate}[(a)]
\item (A paraphrase of) \cite[Question 11]{bg} asks whether the extended derivations of a Lie algebra in characteristic 0 are precisely those of the form
  \begin{equation}\label{eq:extderiv}
    \partial_{a,f}(l') := [f(a),l'],\ \forall f:L\to L',\ \forall l'\in L'
  \end{equation}
  for $a\in L$;
\item And similarly for {\it restricted} Lie algebras in positive characteristic;
\item While \cite[Question 13]{bg} asks whether restricted positive-characteristic Lie algebras have any non-obvious extended endomorphisms (the obvious assigning the identity or, respectively, the zero map on $L'$ to every $L\to L'$).
\end{enumerate}

We answer these in \Cref{th:q11,th:q13}:

\begin{theoremN}
  Let $L$ be either a Lie algebra over a field $\Bbbk$, or a restricted Lie algebra when $\Bbbk$ has positive characteristic.
  \begin{enumerate}[(a)]
  \item If $\mathrm{char}~\Bbbk=0$, the map sending $a\in L$ to the extended derivation $(\partial_{a,f})_f$ of \Cref{eq:extderiv} is an isomorphism between $L$ and the linear space of extended derivations of $L$. 
  \item The same holds in positive characteristic for restricted Lie algebras.
  \item\label{item:done} For Lie algebras (regardless of characteristic) the only extended endomorphisms are the two obvious ones:
    \begin{align*}
      \left(L\to L'\right)&\xmapsto{\quad}\id_{L'}\quad\text{and}\\
      \left(L\to L'\right)&\xmapsto{\quad}0.
    \end{align*}
  \item The analogous statement holds for restricted Lie algebras in positive characteristic.
  \end{enumerate}
  \qedhere
\end{theoremN}

Part \Cref{item:done} is already settled in \cite[Theorem 12]{bg} and is included here only for completeness; the other three items answer the various questions of \cite[\S 8]{bg} indicated above.

\Cref{se:cpctg} focuses on another instance of this same phenomenon, whereby functoriality begets innerness, but this time working in the category $\cat{CGp}_0$ of compact connected topological groups (here always assumed Hausdorff). The partial analogue of \cite[Corollary 3]{bg} is \Cref{th:trivinner}, and reads

\begin{theoremN}
  For a compact connected group $G$, the individual components of a natural endomorphism of
  \begin{equation*}
    G\downarrow \cat{CGp}_0 \xmapsto{\quad U_{G,\cat{CGp}_0}\quad} \cat{CGp}_0
  \end{equation*}
  are either all trivial or all inner automorphisms.  \qedhere
\end{theoremN}

While very similar in character to the results discussed above, the proofs are by necessity quite different. \cite{bg} and the citing literature (e.g. \cite{hps,prkr-pshvs,prkr-racks}) tend to adopt universal-algebra-flavored approaches: the idea is to study the effect of (say) a natural endomorphism of \Cref{eq:fgt} on the morphism
\begin{equation*}
  c\to \langle c,x\rangle
\end{equation*}
into the structure (set, group, etc.) that {\it freely adjoins} an element $x$. Such universal constructions do exist in the category $\cat{CGp}$ of compact groups; in category-theoretic language, $\cat{CGp}$ is, for instance, {\it cocomplete} \cite[Definition 12.2]{ahs}. To construct the {\it coproduct} \cite[\S 10.63]{ahs} (or {\it free product}) $G_1*G_2$ of two compact groups one must
\begin{itemize}
\item form the usual group coproduct $G*_{discrete}G_2$;
\item equip it with the coarsest group topology making the canonical embeddings
  \begin{equation*}
    \iota_i:G_i\to G_1*_{discrete}G_2
  \end{equation*}
  continuous (e.g. \cite[introductory remarks]{ft-free});
\item and then take the {\it Bohr compactification} (\cite[\S III.9]{bjm} or \cite[\S 2.10]{kan-comm}) thereof.
\end{itemize}

It is this last step that disturbs the usual procedure: freely appending an element $x$ to a compact group $G$ amounts to the above with
\begin{equation*}
  G_1=G
  \quad\text{and}\quad
  G_2=\text{Bohr compactification of }\bZ\cong \langle x\rangle.
\end{equation*}
Words in $x^{\pm 1}$ and elements of $G$ no longer constitute all of $\langle G,x\rangle$, but rather only a dense subgroup thereof; for that reason, arguments such as those in the proof of \cite[Theorem 1]{bg} are no longer available.

\subsection*{Acknowledgements}

This work is partially supported by NSF grant DMS-2001128.


\section{Lie algebras}\label{se:lalg}

We work with Lie algebras over fields, for which \cite{jc} is an excellent source. Having fixed a field $k$, $L$ will typically denote
\begin{itemize}
\item a Lie algebra unless specified otherwise, or
\item a {\it restricted Lie algebra} in characteristic $p$, in the sense of \cite[\S V.7 Definition 4]{jc} or \cite[Definition 2.3.2]{mon}.
\end{itemize}

We further write

\begin{itemize}
\item $U(L)$ for the universal enveloping algebra of a Lie $k$-algebra (\cite[\S V.1, Definition 1]{jc}), and
\item $U_p(L)$ for the {\it restricted enveloping algebra} of the restricted Lie algebra $L$ in characteristic $p$ ($\overline{U}_L$ in \cite[\S V.7 Theorem 12]{jc}).
\item $Q(L)$ to denote either $U(L)$ or $U_p(L)$, as appropriate (so as to have uniform notation to refer to both cases).
\end{itemize}

Note that both $U(L)$ and $U_p(L)$ are naturally Hopf algebras over $k$ \cite[Example 1.5.4, Definition 2.3.2]{mon} and hence come equipped with counits $\varepsilon$, comultiplications $\Delta$, etc. We will assume basic background on Hopf algebras as covered, for instance, in \cite{mon}. The Hopf algebra structure is uniquely determined by the requirement that the elements $x\in L$ be {\it primitive} \cite[Definition 1.3.4]{mon}, i.e.
\begin{equation*}
  \Delta(x) = x\otimes 1 + 1\otimes x. 
\end{equation*}

\subsection{Derivations}\label{subse:der}

To streamline the statement of \Cref{th:q11}, we introduce some terminology.

\begin{definition}
  Let $k$ be a field, and $L$ either a Lie algebra (in characteristic zero) or a restricted Lie algebra in characteristic $p$, and $\langle L,x\rangle$ the (restricted) Lie algebra freely generated by $L$ and a formal variable $x$.

  An element $a\in Q(L)$
  \begin{itemize}
  \item is {\it constant-less} if it is annihilated by the counit $\varepsilon:Q(L)\to k$ of the Hopf algebra $Q(L)$. 
  \item {\it induces a universal derivative} (or is {\it universally derivative} or {\it a universal derivative}) if, for a formal variable $x$, the commutator $[a,x]\in Q(\langle L,x\rangle)$ belongs to $\langle L,x\rangle$.
  \end{itemize}
\end{definition}

Note that $L\subset Q(L)$ consists of constant-less universal derivatives. The following result gives the converse, answering \cite[Question 11]{bg} negatively.

\begin{theorem}\label{th:q11}
  Let $k$ be a field and $L$ either a Lie algebra over $k$ (in characteristic zero) or a restricted Lie algebra (in characteristic $p$).
  \begin{enumerate}[(a)]
  \item\label{item:1} If $k$ has characteristic zero then the only constant-less universally derivative elements of $U(L)$ are those of $L$.
  \item\label{item:2} Similarly, if $k$ has characteristic $p$ then the only constant-less universally derivative elements of $U_p(L)$ are those of $L$.
  \end{enumerate}
\end{theorem}
\begin{proof}
  The two statements (and proofs) are very similar, so we treat only the first in detail. The phenomenon driving both arguments is the fact that $L$ can be recovered as precisely the space of primitive elements in the Hopf algebra $Q(L)$.

  \Cref{item:1} As hinted above, note first that in characteristic zero, the primitive elements $P(U(G))$ of $U(G)$ (for an arbitrary Lie algebra $G$) are precisely those of $G\subset U(G)$ \cite[Proposition 5.5.3, part 2)]{mon}. Take $G=\langle L,x\rangle$. The hypothesis
  \begin{equation*}
    [a,x] \in \langle L,x\rangle
  \end{equation*}
  (for some constant-less universal derivative $a\in U(L)$) implies that the commutator $[a,x]$ is primitive. This, in turn, implies that $a$ is primitive. To see this, assume otherwise and write
  \begin{equation*}
    \Delta(a) = a\otimes 1 + 1\otimes a + \sum_i a_{i,1}\otimes a_{i,2} 
  \end{equation*}
  where $a_{i,j}$ are
  \begin{itemize}
  \item constant-less elements of $U(L)$,
  \item with at least one non-zero term $a_{i,1}\otimes a_{i,2}$,
  \item and linearly independent $a_{i,2}$ (since we can always group the tensors so as to arrange for this).
  \end{itemize}
  Expanding $\Delta([a,x])$, the resulting term $a_{i,1}\otimes a_{i,2}x$ appears only once, and hence will not cancel. This contradicts the primitivity of $[a,x]$, concluding that indeed $a$ must be primitive. But then, by the already-cited \cite[Proposition 5.5.3, part 2)]{mon}, $a\in L$.

  \Cref{item:2} The argument goes through almost verbatim, the only difference being that this time around we use the fact that in characteristic $p$ the primitive elements $P(U_p(L))$ of $U_p(L)$ are those of $L$. This is not quite what \cite[Proposition 5.5.3, part 3)]{mon} says, but that proof can be adapted. The claim (that $P(U_p(L))=L$) also follows from \cite[Proposition 13.2.3]{swe}.
\end{proof}

\subsection{Endomorphisms}\label{subse:endo}

There is an endomorphism (as opposed to derivation) version of \Cref{th:q11}, which in turn answers \cite[Question 13]{bg}. Before stating it, some more terminology.

\begin{definition}\label{def:univendo}
  Let $L$ be an object of a category $\cC$. A {\it universal endomorphism} of $L$ in $\cC$ is an endomorphism of the forgetful functor
  \begin{equation*}
    L\downarrow \cC\to \cC. 
  \end{equation*}
\end{definition}

These are the {\it extended inner endomorphisms} of \cite[Definition 4]{bg}. The announced universal-endomorphism version of \Cref{th:q11} now reads

\begin{theorem}\label{th:q13}
  Let $k$ be a field and $L$ either a Lie algebra over $k$ (in characteristic zero) or a restricted Lie algebra (in characteristic $p$).
  \begin{enumerate}[(a)]
  \item\label{item:3} The only universal endomorphisms of $L$ in the category $\cat{Lie}_k$ of Lie $k$-algebras are $0$ and $\id$.
  \item\label{item:4} Furthermore, if $k$ has characteristic $p$ then the only universal endomorphisms of $L$ in the category $\cat{Lie}_{k,p}$ of restricted Lie $k$-algebras are $0$ and $\id$.
  \end{enumerate}  
\end{theorem}
\begin{proof}
  As explained in \cite[discussion preceding Theorem 12]{bg}, a universal endomorphism as in the statement is determined by elements $a,b\in Q(L)$, acquiring the expression
  \begin{equation}\label{eq:amb}
    M\ni m\mapsto \varphi(a)m\varphi(b),
  \end{equation}
  where
  \begin{itemize}
  \item $\varphi:L\to M$ is a (restricted) Lie algebra morphism as well as the corresponding morphism $Q(L)\to Q(M)$ it induces, and
  \item it is understood that for all such $\varphi$, the right-hand side of \Cref{eq:amb} belongs to $M$.
  \end{itemize}
  One can package all of this into its universal (or generic) instance: take $\varphi:L\to M$ to be the inclusion
  \begin{equation*}
    L\subseteq \langle L,x\rangle
  \end{equation*}
  for a formal variable $x$, and require that $axb\in \langle L,x\rangle$. Denote by $\varepsilon$ the counit of the Hopf algebra $Q(L)$, and decompose
  \begin{equation*}
    a = \varepsilon(a) + \overline{a},\quad b = \varepsilon(b) + \overline{b}
  \end{equation*}
  for constant-less $\overline{a}$ and $\overline{b}$; the goal is to show that these two latter elements must vanish.
  
  The argument is now similar to that in the proof of \Cref{th:q11}
  \begin{equation*}
    \Delta(\overline{a}) = \overline{a}\otimes 1 + 1\otimes\overline{a} + \cdots
  \end{equation*}
  and similarly for $\overline{b}$, where the missing summands indicated by `$\cdots$' are simple tensors with constant-less tensorands.

  Because $Q(\langle L,x\rangle)$ is the coproduct (over $k$) of $Q(L)$ and $k[x]$, if $\overline{a}$ and $\overline{b}$ are both non-vanishing then the term $\overline{a}\otimes x\overline{b}$ of $\Delta(axb)\in Q(\langle L,x\rangle)^{\otimes 2}$ will not cancel out, contradicting the fact that
  \begin{equation*}
    axb\in \langle L,x\rangle\subset Q(\langle L,x\rangle)
  \end{equation*}
  is primitive, i.e.
  \begin{equation*}
    \Delta(axb) = axb\otimes 1 + 1\otimes axb. 
  \end{equation*}
  It follows that at least one of $a$ and $b$ is scalar. Suppose it is $b$, so that we can absorb the constant into $a$ and work with $ax$ in place of the original $axb$. Now repeat the argument: if $\overline{a}\ne 0$ then the term $\overline{a}\otimes x$ will be present in $\Delta(ax)$, again contradicting the primitivity of $ax$.

  In conclusion $a$ and $b$ are both scalar, as claimed.
\end{proof}

\begin{remark}
  Alternatively, once we find that one of $a$ and $b$ is scalar we can conclude using the fact that, according to \cite[proof of Theorem 12]{bg}, $ba=1$; this was not used above.
\end{remark}

Part \Cref{item:3} of \Cref{th:q13} recovers \cite[Theorem 12]{bg}, while part \Cref{item:4} answers \cite[Question 13]{bg} negatively.

\begin{remark}\label{re:simpcon}
  {\it Finite-dimensional} (rather than arbitrary) Lie algebras are much more interesting, as extended endomorphisms or automorphisms go. 

  By \cite[\S III.6.1, Theorem 1]{bourb-lie-13} the category $\cat{LAlg}_{f,\Bbbk}$ of finite-dimensional Lie algebras over the real or complex field $\Bbbk$ is equivalent to that of {\it simply-connected} Lie groups over $\Bbbk$. Consequently, for any finite-dimensional Lie algebra $L\in \cat{Lalg}_{f,\Bbbk}$, the corresponding simply-connected Lie group $G_L$ with Lie algebra $L$ gives a wealth of ``inner'' automorphisms of the forgetful functor
  \begin{equation*}
    L\downarrow \cat{Lalg}_{f,\Bbbk}\xrightarrow{\quad U_{L}}\cat{Lalg}_{f,\Bbbk}:
  \end{equation*}
  an element $g\in G_L$ operates on the codomain $L'$ of a morphism
  \begin{equation}\label{eq:fll}
    L\xrightarrow{\quad f\quad}L'\text{ in }\cat{Lalg}_{f,\Bbbk}
  \end{equation}
  via the adjoint action of $\widetilde{f}(g)$, where
  \begin{equation*}
    G_L\xrightarrow{\quad \widetilde{f}\quad}G_{L'}
  \end{equation*}
  is the lift of \Cref{eq:fll} to simply-connected Lie groups (once more, \cite[\S III.6.1, Theorem 1]{bourb-lie-13}).
\end{remark}

\section{Compact groups}\label{se:cpctg}

Some notation:
\begin{itemize}
\item $\cat{CGp}$ the category of compact (Hausdorff) topological groups;
\item $\cat{CGp}_0$ that of {\it connected} compact groups;
\item And in general, for an object $c\in \cC$ and a full subcategory $\cC'\subseteq \cC$, we write
  \begin{equation}\label{eq:ucc}
    c\downarrow \cC'\xrightarrow{\quad U_{c,\cC'}\quad} \cC'
  \end{equation}
  for the respective forgetful functor (as in \Cref{eq:fgt}, with $\cC'$ in place of $\cC$).
\end{itemize}

The present section is concerned with the following (almost) ``automatic innerness'' for extended automorphisms of compact connected groups.

\begin{theorem}\label{th:trivinner}
  Let $G\in \cat{CGp}_0$ be a compact connected group and $\alpha\in \End(U_{G,\cat{CGp}_0})$ a natural endomorphism of the forgetful functor defined as in \Cref{eq:ucc}, assigning an endomorphism $\alpha_f\in \End(H)$ to
  \begin{equation}\label{eq:genericf}
    f:G\to H,\quad H\text{ compact connected}.
  \end{equation}
  One of the two following possibilities obtains:
  \begin{itemize}
  \item $\alpha$ is trivial, in the sense that $\alpha_{f}$ is the trivial endomorphism of $H$ for every \Cref{eq:genericf};
  \item or every $\alpha_f$ is an inner automorphism of $H$. 
  \end{itemize}
\end{theorem}

The proof requires some preparation.

\begin{remark}\label{re:enoughid}
  Consider a full subcategory $\cC'\subseteq \cC$ (as in \Cref{eq:ucc}). The application
  \begin{equation*}
    \cC\ni c\xmapsto{\quad} \End(U_{c,\cC'})
  \end{equation*}
  is functorial (albeit taking values, in principle, in the category of set-theoretically {\it large} monoids).
  
  In the sequel we use this repeatedly, and mostly tacitly. This functoriality will usually make an appearance (with, say, $\cC=\cC'=\cat{CGp}$ or $\cC=\cC'=\cat{CGp}_0$) in the following form: given an endomorphism
  \begin{equation*}
    \alpha\in \End(U_{c,\cC'})
  \end{equation*}
  and a morphism $f:c\to d$ with codomain $d\in \cC'$, the component $\alpha_f$ is itself the identity component
  \begin{equation*}
    \beta_{\id}=\alpha_f\text{ of some }\beta\in \End(U_{d,\cC'})
  \end{equation*}
  induced by $\alpha$.
\end{remark}

\begin{proposition}\label{pr:endoiso}
  \begin{enumerate}[(a)]
  \item\label{item:cgp0} For a compact connected group $G\in \cat{CGp}_0$, a non-trivial natural endomorphism of the forgetful functor
    \begin{equation*}
      G\downarrow\cat{CGp}_0\xrightarrow{\quad U_{G,\cat{CGp}_0}\quad} \cat{CGp}_0
    \end{equation*}
    is automatically a natural automorphism.
  \item\label{item:cgp} The same goes for the forgetful functor $U_{G,\cat{CGp}}$. 
  \end{enumerate}  
\end{proposition}
\begin{proof}
  We focus on \Cref{item:cgp0} to fix ideas; the other proof is largely parallel.

  \begin{enumerate}[(1)]
  \item\label{item:trivinj} {\bf Each $\alpha_f$ is either trivial or injective.} Per \Cref{re:enoughid}, we may as well take $f=\id$.

    Let $g\in G$ be an element not annihilated by $\alpha_{\id}$, and consider a morphism
    \begin{equation}\label{eq:fcgp}
      G\xrightarrow{\quad f\quad}H\text{ in }\cat{CGp}. 
    \end{equation}
    chosen judiciously (more on this momentarily). We have a commutative diagram
    \begin{equation}\label{eq:gghh}
      \begin{tikzpicture}[auto,baseline=(current  bounding  box.center)]
        \path[anchor=base] 
        (0,0) node (l) {$G$}
        +(2,.5) node (u) {$G$}
        +(2,-.5) node (d) {$H$}
        +(4,0) node (r) {$H$.}
        ;
        \draw[->] (l) to[bend left=6] node[pos=.5,auto] {$\scriptstyle \alpha_{\id}$} (u);
        \draw[->] (u) to[bend left=6] node[pos=.5,auto] {$\scriptstyle f$} (r);
        \draw[->] (l) to[bend right=6] node[pos=.5,auto,swap] {$\scriptstyle f$} (d);
        \draw[->] (d) to[bend right=6] node[pos=.5,auto,swap] {$\scriptstyle \alpha_{f}$} (r);
      \end{tikzpicture}
    \end{equation}
    Suppose now that
    \begin{itemize}
    \item $f(\alpha_{\id}(g))$ is non-trivial;
    \item as is $f(g')$ for an arbitrary $1\ne g'\in \ker\alpha_{\id}$;
    \item and $H$ is a compact, connected, {\it simple} Lie group: one with no non-trivial proper normal subgroups or equivalently \cite[Theorem 9.90]{hm4}, no such subgroups that are {\it closed}.
    \end{itemize}
    That such $f$ exist is easily seen, and relegated to \Cref{le:enoughsimple}. The upper path in \Cref{eq:gghh} fails to annihilate $g\in G$, hence so does the lower. $H$ being simple, $\alpha_{f}$ must be one-to-one (because it cannot be trivial). But this means that the lower composition in \Cref{eq:gghh} fails to annihilate $g'$, contradicting the fact that the upper path does.

    The contradiction stems from our assumption that there are non-trivial $g'\in \ker\alpha_{\id}$, so that map must in fact be injective.

  \item\label{item:trivsurj} {\bf Each $\alpha_f$ is either trivial or surjective.} Once more, take $f=\id$. Assume, this time around, that $\alpha_{\id}$ is neither trivial nor onto.

    The subgroup $\alpha_{\id}(G)\le G$ is then proper but non-trivial. Because compact groups are inverse limits of their Lie quotients \cite[Corollary 2.36]{hm4}, we can find such a quotient \Cref{eq:fcgp} so that the image
    \begin{equation}\label{eq:propinh}
      \alpha_f(H) = f(\alpha_{\id}(G))\le H = f(G)
    \end{equation}
    is both proper and non-trivial (the first equality follows from the commutativity of \Cref{eq:gghh}). But then
    \begin{itemize}
    \item $\alpha_f$ cannot be trivial, since its image is not;
    \item hence must be injective by part \Cref{item:trivinj};
    \item and thus also surjective, because it is a one-to-one map of compact {\it Lie} groups: it restricts to an isomorphism on every connected component for dimension reasons, and a compact Lie group has finitely many connected components.
    \end{itemize}
    This latter remark, though, contradicts the properness of \Cref{eq:propinh}.

  \item {\bf Finishing the proof.} This, so far, shows that the {\it individual} components $\alpha_f$ of $\alpha$ are each either trivial or bijective. It remains to argue that we cannot have a mixture of these: if $\alpha_{\id}$ is trivial (bijective) then so, respectively, is every $\alpha_f$.
    
    Consider the commutative diagram
    \begin{equation}\label{eq:6diag}
      \begin{tikzpicture}[auto,baseline=(current  bounding  box.center)]
        \path[anchor=base] 
        (0,0) node (lu) {$G$}
        (0,-1) node (ld) {$G$}
        (2,.5) node (mu) {$G\times H$}
        (2,-1.5) node (md) {$G\times H$}
        (4,0) node (ru) {$H$}
        (4,-1) node (rd) {$H$}
        ;
        \draw[->] (lu) to[bend left=6] node[pos=.5,auto] {$\scriptstyle \iota$} (mu);
        \draw[->] (mu) to[bend left=6] node[pos=.5,auto] {$\scriptstyle \pi$} (ru);
        \draw[->] (ld) to[bend right=6] node[pos=.5,auto,swap] {$\scriptstyle \iota$} (md);
        \draw[->] (md) to[bend right=6] node[pos=.5,auto,swap] {$\scriptstyle \pi$} (rd);
        \draw[->] (lu) to[bend left=0] node[pos=.5,auto,swap] {$\scriptstyle \alpha_{\id}$} (ld);
        \draw[->] (mu) to[bend left=0] node[pos=.5,auto] {$\scriptstyle \alpha_{\iota}$} (md);
        \draw[->] (ru) to[bend left=0] node[pos=.5,auto] {$\scriptstyle \alpha_{\cat{triv}}$} (rd);
      \end{tikzpicture}
    \end{equation}
    (with $\iota$ and $\pi$ the obvious inclusion and projection).

    By claims \Cref{item:trivinj} and \Cref{item:trivsurj}, if at least one of $\alpha_{\id}$ and $\alpha_{\cat{triv}}$ is non-trivial, then $\alpha_{\iota}$ must be bijective. This, in turn, would entail both
    \begin{itemize}
    \item the injectivity of $\alpha_{\id}$;
    \item and the surjectivity of $\alpha_{\cat{triv}}$. 
    \end{itemize}
    
    Suppose first that $\alpha_{\id}$ is trivial. By \Cref{eq:gghh}, the only way for $\alpha_f$ to be non-trivial (and hence injective by \Cref{item:trivinj}) is for $f:G\to H$ itself to be trivial, so that $\alpha_f$ is the $\alpha_{\cat{triv}}$ of \Cref{eq:6diag}.

    Similarly, \Cref{eq:gghh} and the assumed bijectivity of $\alpha_{\id}$ force $f=\cat{triv}$ whenever $\alpha_f$ is trivial; we are then again in the context of \Cref{eq:6diag}.

    As just observed, then, either case entails either the injectivity of one of the maps $\alpha_{\id}$ or the surjectivity of $\alpha_f$, so that map's ($\alpha_{\id}$ or $\alpha_f$) triviality also forces the respective group to be trivial and hence the map in question is also an automorphism.

    Summary: if one of $\alpha_{\id}$ and $\alpha_{f}$ is trivial while the other is bijective, the maps are both bijective; this concludes the proof.
  \end{enumerate}
\end{proof}

The following result is presumably well known, but we isolate it here for reference in the proof of \Cref{pr:endoiso}.

\begin{lemma}\label{le:enoughsimple}
  For any finite set $F\subseteq G$ of non-trivial elements of a compact group there is a morphism $f:G\to PSU(n)$ to some projective special unitary group with $f(g)\ne 1$ for $g\in F$.

  If $G$ is Lie then $f$ can be chosen injective.
\end{lemma}
\begin{proof}
  because $G$ is the inverse limit of its Lie quotient groups \cite[Corollary 2.36]{hm4}, we certainly have a morphism $f$ to a {\it unitary} group with these properties; it is thus enough to assume that $G=U(m)$.
  

  Next, $U(m)$ further embeds into $SU(2m)$ via
  \begin{equation*}
    U(m)\ni x
    \mapsto
    \begin{pmatrix}
      x&0\\
      0&\overline{x}
    \end{pmatrix}
    \in SU(2m)
  \end{equation*}
  (with the overline denoting complex conjugation), so that we can in fact set $G=S(N)$.

  Denoting by $\rho$ the $N$-dimensional defining representation of $SU(N)$, the representation $\rho\oplus \rho^{\otimes 2}$ gives an embedding
  \begin{equation*}
    SU(N)\xrightarrow{\quad} SU(n),\ n:=N+N^2
  \end{equation*}
  with the property that the center $\bZ/N\subset SU(N)$ intersects that of $SU(n)$ trivially, so that further surjecting onto $PSU(n)$ finally gives the desired embedding (of $G$, now assumed Lie, into $PSU(n)$). 
\end{proof}

We will take it for granted that the automorphism groups of the unitary $U(n)$ are
\begin{equation}\label{eq:autun}
  \Aut(U(n))\cong PSU(n)\rtimes \bZ/2,
\end{equation}
where
\begin{itemize}
\item the {\it projective} special unitary group $PSU(n)$ is
  \begin{equation*}
    PSU(n) = U(n)/(\text{central circle }\bS^1)\cong SU(n)/(\text{central }\bZ/n)
  \end{equation*}
  is the {\it inner} automorphism group, acting on $U(n)$ by conjugation;
\item and the generator of $\bZ/2$ is complex conjugation. 
\end{itemize}

This is fairly standard, though a statement specifically to this effect seems difficult to locate in the literature. The proof is certainly no more difficult that that of the analogous result for $SU(n)$, which in turn follows from the classification of the automorphisms of the complexified Lie algebra
\begin{equation*}
  \mathfrak{sl}(n) = \mathfrak{su}(n)\otimes_{\bR}\bC,
\end{equation*}
described, say, in \cite[\S IX.5]{jc}.

\begin{proposition}\label{pr:isinn}
  Let $G\in \cat{CGp}_0$ be a compact connected group and $\alpha$ an automorphism of the forgetful functor $U_{G,\cat{CGp}_0}$

  For $f:G\to H$ the automorphism $\alpha_f\in \Aut(H)$ is inner. 
\end{proposition}
\begin{proof}
  Note that it is enough to prove $\alpha_{\id}$ itself inner: every $\alpha_f$ as in the statement is the identity component of an automorphism of $U_{H,\cat{CGp}_0}$ induced by $\alpha$.

  The proof runs through a number of intermediate steps.
  \begin{enumerate}[(1)]
  \item {\bf The circle: $G=\bS^1$.} In this case $\alpha_{\id}$ is either the identity (which we claim is the case) or the other automorphism: $z\mapsto z^{-1}$. Consider, now, the central embedding
    \begin{equation*}
      \bS^1\xrightarrow[\cong]{\quad f\quad} (\text{diagonal matrices})\subset U(n)
    \end{equation*}
    into a unitary group and the corresponding element $\alpha_f$ of $\Aut(U(n))$.

    
    Because the image of $f$ is central in $U(n)$, further composition with an arbitrary conjugation
    \begin{equation*}
      Ad_u:=u(-)u^*\in \Aut(U(n)),\ u\in U(n)
    \end{equation*}
    will leave $f$ invariant, so the naturality of $\alpha$ gives a commutative diagram
    \begin{equation*}
      \begin{tikzpicture}[auto,baseline=(current  bounding  box.center)]
        \path[anchor=base] 
        (0,0) node (l) {$U(n)$}
        +(2,.5) node (u) {$U(n)$}
        +(2,-.5) node (d) {$U(n)$}
        +(4,0) node (r) {$U(n)$.}
        ;
        \draw[->] (l) to[bend left=6] node[pos=.5,auto] {$\scriptstyle \alpha_{f}$} (u);
        \draw[->] (u) to[bend left=6] node[pos=.5,auto] {$\scriptstyle Ad_u$} (r);
        \draw[->] (l) to[bend right=6] node[pos=.5,auto,swap] {$\scriptstyle Ad_u$} (d);
        \draw[->] (d) to[bend right=6] node[pos=.5,auto,swap] {$\scriptstyle \alpha_{f}$} (r);
      \end{tikzpicture}
    \end{equation*}
    In other words, $\alpha_f$ lies in the centralizer of $PSU(n)$ in \Cref{eq:autun}. That centralizer is easily seen to be trivial, so that $\alpha_f=\id$ and $\alpha_{\id}=\id$ by the commutativity of \Cref{eq:gghh}. 
    
  \item {\bf Unitary groups.} We are assuming now that $G = U(n)$. The preceding point shows that the automorphism $\alpha_{\det}$ associated to the determinant morphism
    \begin{equation*}
      U(n)\xrightarrow{\quad\det\quad}\bS^1
    \end{equation*}
    is trivial. But then, once more by the commutativity of \Cref{eq:gghh} (with $G=U(n)$, $H=\bS^1$ and $f=\det$), this means that $\alpha_{\id}\in \Aut(U(n))$ fixes determinants (rather than inverting them), and hence must be inner by \Cref{eq:autun}. 

  \item {\bf Arbitrary connected compact $G$.} We know from the preceding discussion that every
    \begin{equation*}
      \alpha_f\in \Aut(U(n)),\quad f:G\to U(n)
    \end{equation*}
    induced by $\alpha$ is inner, i.e. conjugation by some $u_f\in U(n)$. But then the commutativity of
    \begin{equation*}
      \begin{tikzpicture}[auto,baseline=(current  bounding  box.center)]
        \path[anchor=base] 
        (0,0) node (l) {$G$}
        +(2,.5) node (u) {$G$}
        +(2,-.5) node (d) {$U(n)$}
        +(4,0) node (r) {$U(n)$.}
        ;
        \draw[->] (l) to[bend left=6] node[pos=.5,auto] {$\scriptstyle \alpha_{\id}$} (u);
        \draw[->] (u) to[bend left=6] node[pos=.5,auto] {$\scriptstyle f$} (r);
        \draw[->] (l) to[bend right=6] node[pos=.5,auto,swap] {$\scriptstyle f$} (d);
        \draw[->] (d) to[bend right=6] node[pos=.5,auto,swap] {$\scriptstyle \alpha_{f}$} (r);
      \end{tikzpicture}
    \end{equation*}
    means that $u_f$, regarded as an automorphism of the carrier space $V\cong \bC^n$ of the $G$-representation $f$, implements an isomorphism between $f$ and the $\alpha_{\id}$-twisted representation $f\circ\alpha_{\id}$.

    This holds for arbitrary representations $f$ of $G$, so $\alpha_{\id}$ is an automorphism of the latter leaving invariant (the isomorphism classes of) all of its irreducible representations. For {\it connected} $G$ this implies that $\alpha_{\id}$ must be inner \cite[Corollary 2]{mcm-dual}.
  \end{enumerate}
  This concludes the proof.
\end{proof}

\begin{remark}\label{re:needconn}
  While I do not know whether \Cref{th:trivinner} holds for disconnected compact groups, the argument in the proof of \Cref{pr:isinn} certainly does not go through: it uses the connectedness of $G$ crucially, in concluding via \cite[Corollary 2]{mcm-dual} that automorphisms that preserve the isomorphism classes of all (or equivalently, all irreducible) representations are inner.

  For {\it finite} groups, for instance, the automorphisms with this preservation property are precisely the ones termed {\it class-preserving} in the rich literature on the topic: an equivalent characterization is that they leave every conjugacy class invariant. 

  The reader can consult, for instance, \cite{bm-class} and their references (\cite{yad-surv,wall}, and so on) for extensive discussions and examples of finite groups which admit {\it outer} class-preserving automorphisms.
\end{remark}

At this point, not much is left to do:

\pf{th:trivinner}
\begin{th:trivinner}
  By \Cref{pr:endoiso}, the statement reduces to the already-proven \Cref{pr:isinn}.
\end{th:trivinner}



\addcontentsline{toc}{section}{References}

\Addresses

\end{document}